\documentclass[12pt,reqno]{amsart}

\usepackage{amsmath,amsfonts,amsbsy,amsgen,amscd,mathrsfs,amssymb,amsthm}
\usepackage{enumerate,mathtools}

\usepackage[usenames,dvipsnames]{xcolor}
\usepackage[colorlinks=true,citecolor=blue,linkcolor=blue]{hyperref}

\usepackage{tikz}
\usetikzlibrary{shadings}

\newtheorem{thm}{Theorem}[section]
\newtheorem{lem}[thm]{Lemma}

\newtheorem{rem}[thm]{Remark}

\newcommand{\pd}{\partial}

\theoremstyle{definition}

\numberwithin{equation}{section} 
\numberwithin{figure}{section}
\numberwithin{table}{section}

\newcommand{\bT}{\mathbf{T}}
\newcommand{\bD}{\mathbf{D}}
\newcommand{\bE}{\mathbf{E}}
\newcommand{\eps}{\varepsilon}
\newcommand{\bR}{\mathbf{R}}
\newcommand{\bC}{\mathbf{C}}

\newcommand{\Om}{\Omega}

\newcommand{\al}{\alpha}

\newcommand{\la}{\lambda}

\newcommand{\cP}{\mathcal{P}}
\newcommand{\cT}{\mathcal{T}}
\newcommand{\cR}{\mathcal{R}}

\newcommand{\vf}{\varphi}

\begin{document}

\title[Tail spaces and Bernstein--Markov inequality]{Tail spaces estimates on Hamming cube and Bernstein--Markov inequality}




\author{Alexander Volberg}
\address{(A.V.) Department of Mathematics, MSU, 
East Lansing, MI 48823, USA }
\email{volberg@math.msu.edu}

\begin{abstract}
This note contains some  estimates for tail  spaces on Hamming cube. We use analytic paraproduct operator for that purpose. We also show several types of Bernstein--Markov inequalities  for Banach space valued functions on Hamming cube. Here the novelty is in getting rid of some irritating logarithms and in proving Bernstein--Markov inequalities for $|\nabla f|_X$  rather than for $\Delta^{1/2} f$  for $X$-valued polynomials on Hamming cube.
\end{abstract}

\thanks{ The research of the  author is supported by NSF DMS-1900286, DMS-2154402 and by Hausdorff Center for Mathematics }



\maketitle


\section{Introduction}

We are interested in tail spaces $\cT_d(X)$ of functions 
$$
f=\sum_{|S|>d} \hat f(S) \eps^S
$$
defined on the Hamming cube $\Om_n=\{ \eps=(\eps_1, \dots, \eps_n) \in \{-1,1\}^n\}$. Here $\hat f (S)$ are coefficients belonging to Banach space $X$.

Later we will also need the space of polynomials $\cP_d(X)$:
$$
f=\sum_{|S|\le d} \hat f(S) \eps^S
$$
defined on the Hamming cube $\Om_n$. Here $\hat f (S)$ are coefficients belonging to Banach space $X$, which, in particular can be just $\bR$.

\bigskip

\noindent{\bf Acknowledgements.}
I am grateful to Alexander Borichev for very valuable discussions. I am also grateful to Alexandros Eskenazis who found a mistake in the first version of Theorem \ref{RXf}.

\section{Tail spaces estimates and dependence on $K$-convexity and $d$}
\label{tail}
We first cite  the problem.
In \cite{MN}, Mendel and Naor asked if for every $K$-convex Banach space $X$ and $p\in(1, \infty)$ there exists a finite positive $c(p,X)$ such that for every $n$ and $d<n$, every function $f:\Om_n\to X$ in the $\cT_d(X)$ satisfies the estimate
$$
\|\Delta f\|_{L^p(X)} \ge c(p, X) d \|f\|_{L^p(X)}\,.
$$
Moreover, they formulated the tail smoothing conjecture: for every $K$-convex Banach space $X$ and $p\in(1, \infty)$ there exists a finite positive $c(p,X), C(p,X)$ such that for every $n$ and $d<n$, every function $f:\Om_n\to X$ in the $\cT_d(X)$ satisfies the estimate
$$
\|P_t f\|_{L^p(X)} \le c(p, X) e^{-C(p, X) d\,t} \|f\|_{L^p(X)}\,,
$$
where $P_t=e^{-t\Delta}$ is the heat semi-group on Hamming cube. The estimate for $\|\Delta f\|_{L^p(X)} $ would follow just by integrating from $0$ to $\infty$.
In Theorem 5.1 of \cite{MN} they proved that there exists $A(p, X)\ge 1$ such that
$$
\|P_t f\|_{L^p(X)} \le c(p, X) e^{-C(p, X) d\,\min(t, t^{A(p, X)})} \|f\|_{L^p(X)}\,.
$$
Before continuing let us cite  a crucial result of Pisier \cite{PAnnals}.

\begin{thm}[Pisier]
\label{pisier}
Let $X$ be a $K$-convex space. Then for any $p\in (1, \infty)$ there exists angle $\alpha \pi\in (0, \pi)$  such that 
operators $e^{-z \Delta}$ are well defined and uniformly bounded in $A_\al:=\{z\in \bC: |\arg z|\le \frac{\al \pi}{2}\}$.
\end{thm}

In \cite{EI1} the following results were proved:

\begin{thm}[Eskenazis--Ivanisvili]
\label{eps}
Let $X$ be a $K$-convex Banach space and let its angle be provided by Theorem \ref{pisier}. Then for every $\eps>0$ we have
$$
\|P_t f\|_{L^p(X)} \le c(p, X, \eps) e^{-C(p, X, \eps) d\,\min(t, t^{\frac1\al+\eps})} \|f\|_{L^p(X)}\,,
$$
$$
\|\Delta f\|_{L^p(X)} \ge c(p, X, \eps) d^{\al-\eps} \|f\|_{L^p(X)}\quad \forall \eps>0.
$$
\end{thm}

\begin{rem}
In \cite{MN} it was proved that in fact
$$
\|P_t f\|_{L^p(X)} \le c(p, X) e^{-C(p, X) d\,\min(t, t^{\frac1\al})} \|f\|_{L^p(X)}\,,
$$
$$
\|\Delta f\|_{L^p(X)} \ge c(p, X) d^{\al} \|f\|_{L^p(X)}\,.
$$
We will prove differently the second inequality in Theorem \ref{KXal} below.  In the last section we explain how one can slightly change the reasoning of \cite{EI1} to
get rid of $\eps$ as well.
\end{rem}

\bigskip

\begin{thm}[Eskenazis--Ivanisvili]
\label{KXdm}
Let $f\in \cP_{d+m}(X)\cap \cT_d(X)$.  Let $X$ be a $K$-convex Banach space. Then
$$
\|\Delta f\|_{L^p(X)} \ge c(p, X) \frac{d}{m} \|f\|_{L^p(X)}\,.
$$
Moreover,
$$
\|\Delta^{1/2} f\|_{L^p(X)} \ge c(p, X) \Big(\frac{d}{m}\Big)^{1/2} \|f\|_{L^p(X)}\,.
$$
\end{thm}

\bigskip

It is well known that the second inequality above implies the first one, just because \cite{FLP}
$$
\|\Delta^{\beta} f||_{L^p(X)} \le 4 \|\Delta f||_{L^p(X)} ^\beta \cdot \|f\|_{L^p(X)}^{1-\beta}\,.
$$

Now we will prove the result that falls a bit short of the conjecture, but at least it gets rid of $1/m$ in the Theorem \ref{KXdm} above and gets rid of $\eps$ in Theorem \ref{eps}.

\bigskip

\begin{thm}
\label{KXal}
Let $X$ be a $K$-convex space, $p\in (1, \infty)$, and $\al$ from Pisier theorem be in $(0,1]$. Then 
$$
\|\Delta f\|_{L^p(X)} \ge c(p, X) d^\al \|f\|_{L^p(X)}\,.
$$
\end{thm}

\begin{rem}
This theorem is not new, see \cite{MN} Theorem 5.1. But the proof is different and it uses the so-called analytic paraproduct operators, about which a lot is known.  Also in the last section we will show how to get rid of $\eps$ easily in Theorem \ref{eps} above \textup(=Theorem 6 of \cite{EI1}\textup).
\end{rem}

\begin{proof}

Given $f$ on $\Om_n$ from the tail space $\cT_d(X)$ let us introduce a new function $F$ of one more variable:
$$
F(w, \eps) := \sum_{|S|>d} w^{|S|} \hat f(S) \eps^S,
$$
and let us recognize that it is a bounded $L^p(X)$-valued function of $w$ in $2$-gone $G_\al:=\{ w: w= e^{-z}, z\in A_\al\}$. This is just reformulation of Pisier's theorem. Moreover it is not only bounded but also holomorphic in $G_\al$, 
$$
F\in H^{\infty}(G_\al; L^p(X))\,.
$$
Notice that the same works for $-G_\al$, as the flip $w\to -w$ can be absorbed by the flip $\eps\to -\eps$. 
Let us consider a domain $O_\al:=G_\al\cup-G_\al$. It is easy to see that its boundary is smooth at all points except $-1$ and $1$, where $O_\al$ forms angle $\pi\al$ by two real analytic curves (actually its boundary is real analytic except for  $\pm 1$, where real analytic curves form an angle $\al\pi$). Operator
$$
 \sum_{|S|}  \hat f(S) \eps^S\to  \sum_{|S|} w^{|S|} \hat f(S) \eps^S
 $$
 is uniformly bounded from $L^p(X)$ to itself by Theorem \ref{pisier}.

 \bigskip
 
 Notice that $wF'_w= \sum_{S} w^{|S|} |S| \hat f (S) \eps^S$ and $\Delta f=  \sum_{S} |S| \hat f (S) \eps^S$. So $wF'(w)$ is  obtained by applying Pisier's Fourier multiplier operator  to $\Delta f$. Therefore we have

\begin{equation}
\label{dF}
\|F'_w(w)\|_{H^\infty(O_\al, L^p(X))} \le C\|wF'_w(w)\|_{H^\infty(O_\al, L^p(X))} \le M \|\Delta  f\|_{L^p(X)}\,.
\end{equation}
The first inequality is a trivial maximal principle for holomorphic functions, the second one is again the same Pisier's theorem.

\bigskip

By the next step we  want the estimate of the following type (for $F'(w)$ built by $f\in \cT_d(X)$)
\begin{equation}
\label{FbydF}
\|F(w)\|_{H^\infty(O_\al, L^p(X))} \le \eps_d \|F'_w(w)\|_{H^\infty(O_\al, L^p(X))} 
\end{equation}
with a small $\eps_d$ for large $d$.

\bigskip

Let $\vf$ be the conformal map from the unit disc $\bD$ onto $O_\al$, $\vf(0)=0$.
Then \eqref{FbydF} is equivalent to

\begin{equation}
\label{FbydF1}
\|F\circ \vf\|_{H^\infty(\bD, L^p(X))} \le \eps_d \|F'\circ \vf\|_{H^\infty(\bD, L^p(X))} 
\end{equation}

Obviously, for $z\in \bD$
$$
F\circ \vf (z) =\int_0^z F'\circ \vf(\zeta) \cdot \vf'(\zeta) d\zeta\,.
$$
Denote $g(\zeta):= F'\circ \vf(\zeta)$ and introduce the operator (sometimes called {\it analytic paraproduct with symbol $\vf$})
$$
T_{\vf} g:= \int_0^z g(\zeta) \cdot \vf'(\zeta) d\zeta\,.
$$
Seems like it was Pommerenke who studied this operator first. Then it  was widely researched, see e.g. \cite{Po}, \cite{CPPR}, \cite{SSV} and the references therein. 
We need to understand the estimate of this operator on the space  of (vector-valued) functions with the property that all their Taylor coefficients at $0$ 
vanish till the order $d-1$. In other words we need to understand the norm of the operator $T_{\psi}$, where $\psi=\int_0^\zeta \zeta^{d-1} \vf'(\zeta)\, d\zeta$.

Let us write the Taylor expansion of conformal map $\vf$:
$$
\vf(z)=c_1^\al z+ c_2^\al z^2 +\dots\,.
$$
Then it is possible to prove
\begin{lem}
\label{coef}
\begin{equation}
|\vf'(z)| \asymp \frac1{|1-z^2|^{1-\al}}\,.
\end{equation}
And  $|c_n^\al | \asymp n^{-1-\al}$.
\end{lem}
See Section \ref{add} for the proof.

\bigskip

We can write
$$
T_{\vf} g= \int_0^z g(\zeta) \cdot \vf'(\zeta) d\zeta =c_1^\al \int_0^z g d\zeta + 2c_2^\al\int_0^z g\cdot \zeta d\zeta +\dots +
$$
$$
 mc_m^\al \int_0^z g\cdot \zeta^{m-1} d\zeta+\dots\,.
$$

Function $g(\zeta)\cdot \zeta^{m-1}$ in our case is from $\cT_{k}, k\ge d,$ and its $H^\infty(\bD)$ norm is exactly the $H^\infty(\bD)$ norm of $g$ itself.

So we need to understand how to estimate the norm of integration operator on $\cT_{d+m}, m\ge 0,$ for all $m$ in $H^\infty(\bD)$  norm (vector-valued, but this will not ne important).

\begin{lem}
\label{conv}
Let $k$ be positive integer. There exists an $L^1(\bT)$ function $s$  with $L^1(\bT)$  norm  at most $\frac{C_0}{k}$ such that $\hat s(k+j), j\ge 0$, is $\frac1{k+j}$.
\end{lem}

Suppose this Lemma is proved. The integration operator maps $\zeta^{k+j}$ into $\frac{\zeta^{k+j+1}}{k+j+1}$. The convolution operator
$s(\zeta)\star \cdot$ maps $\zeta^{k+j}$ into $\frac{\zeta^{k+j}}{k+j}$. Thus the integration operator is division by $\zeta$ composed with convolution with $s$. But devision on $\zeta$ does not change the norm on tail spaces of $H^\infty(\bD)$ (regardless of whether it is vector valued or scalar valued). Hence,  if Lemma \ref{conv} is proved then  we can estimate $g\in H^\infty(\bD; L^p(X))$, $g(0)=0, g'(0)=0,\dots , g^{(d-1)}(0)=0$:
$$
\|Tg\|_{H^\infty(\bD; L^p(X))} \le C_0\|g\|_{H^\infty(\bD; L^p(X))} \sum_{m=1}^\infty m|c_m^\al | \frac1{d+m-1}\,.
$$
Using Lemma \ref{coef} we conclude
$$
\|Tg\|_{H^\infty(\bD; L^p(X))} \le C\|g\|_{H^\infty(\bD; L^p(X))} \sum_{m=1}^\infty \frac1{m^\al} \frac1{d+m-1}\le
$$
$$
C_\al\|g\|_{H^\infty(\bD; L^p(X))} d^{-\al}\,.
$$

Hence we proved that
\begin{equation}
\label{FbydF2}
\| F \circ \vf \|_{H^\infty(\bD; L^p(X))} \le C_\al\|F'_w\circ \vf\|_{H^\infty(\bD; L^p(X))} d^{-\al}\,.
\end{equation}
Then we get \eqref{FbydF}:
$$
\| F \|_{H^\infty(O_\al; L^p(X))} \le C_\al\|F'_w\|_{H^\infty(O_\al ; L^p(X))} d^{-\al}\,.
$$
We can now combine this with \eqref{dF} and obtain
$$
\| F \|_{H^\infty(O_\al; L^p(X))} \le C_\al\, M \|\Delta F\|_{L^p(X))} d^{-\al}\,.
$$
But $F(1)=f$. Hence we get the proof of the theorem modulo Lemma \ref{conv}:
$$
\| f \|_{L^p(X))} \le C_\al\, M \|\Delta F\|_{L^p(X))} d^{-\al}, \quad \forall f\in \cT_d(X)\,.
$$

\bigskip

\begin{proof}
The proof of Lemma \ref{conv}. We assume that $k$ is even, which is enough for the proof. We consider first $S(x) = \sum_{j=1}^\infty \frac{\sin jx}{j}$, it is bounded and has Fourier coefficients 
as we wish: $1/j$, $j\neq 0$. Now we wish to change its Fourier coefficients in the interval $j\in [-k, k]$ and not change the Fourier coefficients outside this interval, and make the $L^1(-\pi, \pi)$ norm of modified function to be  at most $C_0/k$. Consider nodes $x_r:=\frac{\pi r}{k+1}$, $r=-k+1, \dots, -3, -1, 1, 3, \dots, k-1$. The number of nodes is $k$. Construct the Lagrange trigonometric interpolation polynomial $L_k(x)$,
$$
L_k(x)=\sum_r S(x_r) \frac{\Pi_{m\neq r} (\sin x- \sin x_m)}{\Pi_{m\neq r} (\sin x_r- \sin x_m)}\,.
$$
Clearly $L_k(x)$ has non-zero Fourier coefficients only on $[-k+1, k-1]$.
It is easy to check that it is an odd function. Notice that the sign of $S(x) - L_k(x)$ alternates, it is a positive function on $[0, \frac{\pi}{k+1})$, negative on $(\frac{\pi}{k+1}), \frac{3\pi}{k+1})$, et cetera. 

Consider ``triangular" $\cos$, call it $c(x)$, it is linear on $[-\pi, 0]$, linear on $[0, \pi]$ and $c(0)=1, c(\pm \pi)=-1$.  Its integral vanishes. So if we consider $c((k+1)x)$ its Fourier coefficients vanish on $[-k, k]$. Then $c'((k+1) x)$ also has its Fourier coefficients vanishing on $[-k, k]$, in particular, the scalar product in $L^2(-\pi, \pi)$
$$
\langle L_k(x), c'((k+1) x)\rangle=0\,.
$$
Therefore,
$$
\langle S(x)-L_k(x), c'((k+1) x)\rangle=\langle S(x), c'((k+1) x)\rangle\,.
$$
But $c'((k+1)x=\pm 1$,  moreover the pattern of signs repeats the pattern of signs  of  $S(x) - L_k(x)$, namely it is a positive function on $[0, \frac{\pi}{k+1})$, negative on $(\frac{\pi}{k+1}, \frac{3\pi}{k+1})$, et cetera. 
We conclude that
$$
\|S(x)-L_k(x)\|_1=\langle S(x)-L_k(x), c'((k+1) x)\rangle= \langle S(x), c'((k+1) x)\rangle\,.
$$
But obviously the right hand side here is  at most $C_0/k$ again by noticing that the oscillation of $S$ on intervals of order $\asymp 1/k$ is $c/k$.
Lemma \ref{conv} is proved and thus the tail theorem is proved as well.

\end{proof}

\end{proof}

\begin{rem}
Lemma \ref{conv} should be very well known and widely used, but I am grateful to Rostislav Matveev \cite{RM} for this elegant proof of Lemma \ref{conv}. 
\end{rem}

\section{On Bernstein--Markov inequality and the dependence on $X$ and $p$}
\label{BM}

By $L^p(X)$ we always mean $L^p(\Om_n; X)$.
 
We first cite four theorems from \cite{EI1}.
\begin{thm}[Eskenazis--Ivanisvili]
\label{allX}
Let $X$ be an arbitrary Banach space and $p\in [1, \infty]$. Then
$$
\|\Delta f\|_{L^p(X)} \le d^2 \|f\|_{L^p(X)},\quad \forall f\in \cP_d(X)\,.
$$
\end{thm}
\begin{thm}[Eskenazis--Ivanisvili]
\label{etaX}
Let $X$ be a Banach space and $p\in [1, \infty]$. Then if  for all $n$
$$
\|\Delta f\|_{L^p(X)} \le (1-\eta) d^2 \|f\|_{L^p(X)},\quad \forall f\in \cP_d(X),
$$
then $X$ is of finite co-type.
\end{thm}

\begin{thm}[Eskenazis--Ivanisvili]
\label{KX}
Let $X$ be a $K$-convex Banach space and $p\in (1, \infty)$. Then
$$
\|\Delta f\|_{L^p(X)} \le C(p, X)d^{2-\eps(p, X)} \|f\|_{L^p(X)},\quad \forall f\in \cP_d(X),
$$
\end{thm}

\begin{thm}[Eskenazis--Ivanisvili]
\label{RX}
Let $X=\bR$  and $p\in (1, \infty)$. Then
$$
\|\Delta f\|_{L^p} \le C(p)d^{2- \frac{2}{\pi} \arcsin\frac{2\sqrt{p-1}}{p}} \|f\|_{L^p},\quad \forall f\in \cP_d(X)\,.
$$
\end{thm}

Now we will prove the following results.
\begin{thm}
\label{RXf}
Let $X=\bR$,  $p\in [2, \infty)$.  Then
$$
\|\nabla f\|_{L^p} \le C(p)d^{1-\frac1{\pi} \arcsin\frac{2\sqrt{p-1}}{p}} \|f\|_{L^p},\quad \forall f\in \cP_d(X)\,.
$$
If $p\in (1, 2)$ then 
$$
\|\nabla f\|_{L^p} \le C(p)d^{\frac2p -\frac{2\arcsin\frac{2\sqrt{p-1}}{p}}{p\pi}} \|f\|_{L^p},\quad \forall f\in \cP_d(X)\,.
$$
\end{thm}

\bigskip

\begin{rem}
\label{log}
It is almost Theorem 15 of \cite{EI1}, but we get rid of $\log d$ term in the estimate (30) of Theorem 15 of \cite{EI1}.
\end{rem}

We recall the reader  that if $1<p<\infty$ then $\|\Delta^{1/2} f\|_{L^p} \le C(p) \|\nabla f\|_{L^p}$ for scalar valued functions (the result of  E. Ben Efraim and F. Lust-Piquard \cite{BELP}). But the opposite inequality true only for $2\le p<\infty$,  \cite{BELP}. Morally this means that it is more difficult to estimate from above $|\nabla f |$ than $\Delta^{1/2} f$ (even for scalar valued $f$).
Also clearly the power of $d$ doubles up when we pass the estimate from $\|\Delta^{1/2} f\|_{L^p}$ to the estimate of $\|\Delta f\|_{L^p}$.

\bigskip

When we deal with the $\cP_d(X)$ and $X^*$ has type $2$, it also has finite co-type  $r\in [2, \infty)$ by K\"onig--Tzafriri theorem (see \cite{HVNVW2}, Theorem 7.1.14). Then we have the following result.
\begin{thm}
\label{Xf}
Let   $X^*$ be of type $2$ (and automatically of certain co-type $r<\infty$), and $p\in (1, \infty)$,  $\frac1p+\frac1q=1$, then 
$$
\||\nabla f|_X\|_{L^p(X)} \le C(p)d^{2-\frac{2}{\max(q, r)}} \|f\|_{L^p(X)},\quad \forall f\in \cP_d(X),
$$
where
$$
|\nabla f|_X =\big(\sum_{i=1}^n \|D_i f\|_X^2\big)^{1/2}\,.
$$
\end{thm}

\bigskip

\begin{proof} We prove Theorem \ref{RXf}. 
We recall the formula from \cite{IVHV}:
\begin{equation}
\label{formula}
D_j e^{-t\Delta} f (\eps)= \frac{e^{-t}}{(1-e^{-2t})^{1/2}} \bE_\xi\Big( \delta_j (t) f(\eps_1\cdot\xi_1(t), \dots, \eps_n\cdot\xi_n(t))\Big)\,.
\end{equation}
Here 
$$
\delta_j(t) := \frac{\xi_j(t)-e^{-t}}{(1-e^{-2t})^{1/2}},
$$
where $\xi_j(t)$ are independent random variables having values $\pm 1$ with probabilities $\frac{1\pm e^{-t}}{2}$.

From \eqref{formula} for every $\eps \in \Om_n$ we can write
$$
|\nabla e^{-t\Delta} f(\eps)| =\frac{e^{-t}}{(1-e^{-2t})^{1/2}} \max_{\la: \|\la\|_{\ell^2_n} =1} \Big|\bE_\xi \sum_{j=1}^n \la_j  \delta_j (t) f(\eps\cdot\xi(t))\Big|\,.
$$
Hence,
$$
|\nabla e^{-t\Delta} f(\eps)|\le  \frac{e^{-t}}{(1-e^{-2t})^{1/2}}\max_{\la: \|\la\|_{\ell^2_n} =1}\bE_\xi  \Big|\sum_{j=1}^n \la_j  \delta_j (t) f(\eps\cdot\xi(t))\Big|\le
$$
$$
 \frac{e^{-t}}{(1-e^{-2t})^{1/2}} \max_{\la: \|\la\|_{\ell^2_n} =1}\Big(\bE_\xi  \Big|\sum_{j=1}^n \la_j  \delta_j (t) \Big|^q\Big)^{1/q} \Big(\bE_\xi |f(\eps\cdot\xi)|^p\Big)^{1/p}\,.
$$
Raise it to the power $p$ and integrate:
$$
\||\nabla e^{-t\Delta} f(\eps)|\|_p^p \le \Big(\frac{e^{-t}}{(1-e^{-2t})^{1/2}}\Big)^p \bE_\eps\bE_\xi  |f(\eps\cdot\xi)|^p \cdot   \max_{\la: \|\la\|_{\ell^2_n} =1}\Big(\bE_\xi  \Big|\sum_{j=1}^n \la_j  \delta_j (t) \Big|^q\Big)^{p/q} =
$$
$$
 \bE_\xi\bE_\eps |f(\eps\cdot\xi)|^p \cdot  \Big(\frac{e^{-t}}{(1-e^{-2t})^{1/2}}\Big)^p\cdot \max_{\la: \|\la\|_{\ell^2_n} =1}\Big(\bE_\xi  \Big|\sum_{j=1}^n \la_j  \delta_j (t) \Big|^q\Big)^{p/q}=
$$
$$
\Big(\frac{e^{-t}}{(1-e^{-2t})^{1/2}}\Big)^p\|f\|_p^p \cdot   \max_{\la: \|\la\|_{\ell^2_n} =1}\Big(\bE_\xi  \Big|\sum_{j=1}^n \la_j  \delta_j (t) \Big|^q\Big)^{p/q}\,.
$$

Consider the case $1<q\le 2$. Then we just use 
$$
\Big(\bE_\xi  \Big|\sum_{j=1}^n \la_j  \delta_j (t) \Big|^q\Big)^{p/q}\le \Big(\bE_\xi  \Big|\sum_{j=1}^n \la_j  \delta_j (t) \Big|^2\Big)^{p/2}=1,
$$
because $\{\delta_j(t)\}_{j=1}^n$ is an orthonormal system and $ \|\la\|_{\ell^2_n} =1$.

We will use this later:
\begin{equation}
\label{1q2}
1<q\le 2\Rightarrow \||\nabla e^{-t\Delta} f(\eps)|\|_p^p \le \frac{e^{-t}}{(1-e^{-2t})^{1/2}} \|f\|_p^p
\end{equation}

Now let us consider the case $q>2$. In this case we need to estimate $\Big(\bE_\xi  \Big|\sum_{j=1}^n \la_j  \delta_j (t) \Big|^q\Big)^{1/q}$ differently.  First of all we can replace $\delta_j(t)$ by
$$
\tilde\delta_j(t):=\frac{\xi_j(t)-\xi_j'(t)}{(1-e^{-2t})^{1/2}}
$$
with $\xi_j'(t)$ be an independent copy of $\xi_j(t)$. This is just by Jensen inequality and $\bE\xi_j'(t)=e^{-t}$. Random variables are symmetric and we use the following result. 

The following contraction principle is a classical result of Maurey and 
Pisier (see, e.g., \cite[Proposition 3.2]{P}). We spell out a version 
with explicit constants.

\begin{thm}
\label{mp}
Let $(X,\|\cdot\|)$ be a Banach space of cotype $r<\infty$,
let $\tilde\delta_1,\ldots,\tilde\delta_n$ be i.i.d.\ symmetric random variables,
and let $\varepsilon$ be uniformly distributed on $\{-1,1\}^n$. Then
for any $n\ge 1$, $\la_1,\ldots,\la_n\in X$, and $1\le q<\infty$, we have
$$
	\Bigg(
	\mathbf{E}\Bigg\|\sum_{j=1}^n\la_j\ \tilde\delta_j \Bigg\|^q\Bigg)^{1/q}
	\le
	L_{r,q}
	\int_0^\infty \mathbf{P}\{|\tilde\delta(t)_1|>s\}^{\frac{1}{\max(q,r)}}ds\,
	\Bigg(
	\mathbf{E}\Bigg\|\sum_{j=1}^n \la_j\varepsilon_j \Bigg\|^q
	\Bigg)^{1/q}
$$
with $L_{r,q}=L\,C_q(X)\max(1,(r/q)^{1/2})$,
where $L$ is a universal constant.
\end{thm}

In the current situation $X=\bR$, so $r=2$ and $\max(q, 2)=q$. Notice also that
$$
	\int_0^\infty \mathbf{P}\{|\xi_j(t)-\xi_j'(t)|>s\}^{1/q}ds
	=
	2^{1-1/r} (1-e^{-2t})^{1/q}.
$$
Therefore,
$$
\Big(\bE_\xi  \Big|\sum_{j=1}^n \la_j  \delta_j (t) \Big|^q\Big)^{1/q}\le C (1-e^{-2t})^{1/q-1/2}\Big(\bE_\xi  \Big|\sum_{j=1}^n \la_j  \eps_j \Big|^q\Big)^{1/q} \le 
$$
$$
 \frac{C(q) }{(1-e^{-2t})^{1/2-1/q}}
$$
by Khintchine inequality and  by  $ \|\la\|_{\ell^2_n} =1$.

We will use this later: if $q>2$ then
\begin{equation}
\label{2qq}
\||\nabla e^{-t\Delta} f(\eps)|\|_p \le  \frac{C(q)e^{-t}}{(1-e^{-2t})^{1-1/q}} \|f\|_p=\frac{C(q)e^{-t}}{(1-e^{-2t})^{1/p} }\|f\|_p\,.
\end{equation}

Now let us use \eqref{1q2} for $p\ge 2$ and \eqref{2qq} for $1<p<2$ to finish the proof.
We can consider $x=e^{-t}$ and write those inequalities as the estimate of $p$-th norm of 
$$
F_f(x, \eps):= \sum_S x^{|S|} \hat  f(S)\eps_S, \quad\text{where}\,\, f= \sum_S \hat  f(S)\eps_S, \,\,0\le x\le 1\,.
$$
We get
\begin{equation}
\label{Fp2}
\||\nabla F(x, \cdot)|\|_p \le \frac{|x|}{(1-x^2)^{1/2}} \|f\|_p, \quad \forall x\in [-1, 1], p\ge 2\,
\end{equation}
and
\begin{equation}
\label{F1p}
\||\nabla F(x, \cdot)|\|_p \le \frac{|x|}{(1-x^2)^{1/p}} \|f\|_p, \quad \forall x\in [-1, 1], 1<p<2\,.
\end{equation}
We initially have this estimates only for $0\le x\le1$ but flipping $x\to -x$ is absorbed by flipping $\eps\to -\eps$.
By other methods these estimates were obtained also in \cite{EI1}, see (229) and (203) there.

\bigskip

Now we consider an auxiliary domain of the type considered in in \cite{EI1}.  Let us fix $\beta\in (1, 2)$ to be chosen later. Fix $r>1$.
Consider lens domain $\Om(r)=\{z: |z-i\sqrt{r^2-1}| \le r, \, |z+i\sqrt{r^2-1}|\le r\}$. Consider 
$$
\Om(r, \beta):= \Big(1-\frac1{d^\beta}\Big) \Om (r)\,.
$$

Let $G_{\beta, r}$ denote Green's function with pole at infinity of $\bC\setminus \Om(r, \beta)$. It is rather easy to see that
\begin{equation}
\label{Green}
G_{\beta, r}(1) \asymp d^{-\beta \frac{\pi}{2\pi - 2\arcsin\frac{2\sqrt{p-1}}{p}}},
\end{equation}
(notice that  $2\pi - 2\arcsin\frac{2\sqrt{p-1}}{p}$ is the exterior angle for $\Om(r, \beta)$ at corner points of the lens).

We choose $\beta$ in \eqref{Green} to have $G_{\beta, r}(1) \asymp \frac1d$, that is 
\begin{equation}
\label{beta} 
\beta= 2-\frac2{\pi} \arcsin \frac{2\sqrt{p-1}}{p}\,.
\end{equation}

\bigskip

\subsection{Complex variable}
\label{complex}
Now consider a new function in the complex domain:
$$
H(z):= \log \, \||\nabla F(z, \cdot)|\|_p\,.
$$
Notice that this function is subharmonic in the whole $\bC$. To see this one should  write the norm of the gradient as the
supremum over the dual space $L^q(\Om_n, \ell^2_n)$. 
Then we will get that $ \||\nabla F(z, \cdot)|\|_p$ is the supremum over the unit ball  of this dual space 
of the absolute values of linear combinations of $D_j F(z, \eps)$. 
Each such term is analytic and logarithm of absolute value 
of linear combination of such terms is subharmonic. 
The supremum can be interchanged with logarithm and we get that  $H(z)$ is subharmonic.

Let us collect properties of $H$. As $f$ is a polynomial of degree $d$, we get that the growth of $H$ at infinity is majorized by $d\log |z|$. 

In the other hand, we can always think that $\|f\|_p=1$, and then  we just saw that
on the interval $[-1+\frac1{d^\beta}, 1-\frac1{d^\beta}]$ function $H(x)$ has the estimate:
$$
F(x)/d^{\beta/2} \le 1, \quad\text{if}\,\, p\ge 2;\quad  F(x) /Cd^{\beta/p} \le 1,\quad\text{if}\,\, 1<p<2\,.
$$

Then, say, $H(z) -\frac
\beta{2} \log d$ is non positive on $[-1+\frac1{d^\beta}, 1-\frac1{d^\beta}]$ and is of order $d \log|z|$ at infinity. 

\bigskip

But we can say much more by Weissler \cite{We} and Ivanisvili--Nazarov \cite{IN}. It tuns out that then  $H(z) -\frac
\beta{2} \log d$ is non positive on $\bC\setminus G_{\beta, r}$. These are the complex hypercontractivity results.

\bigskip

Hence,  using Green's function $G_{\beta, r}$ of $\bC\setminus G_{\beta, r}$ with pole at infinity
 we get that
$$
H(z) -\frac
\beta{2} \log d \le dG_{\beta,r}(z)
$$
uniformly in $\bC\setminus G_{\beta, r}$.
Hence,
$$
\frac{\||\nabla F(z, \cdot)|\|_p}{d^{\beta/2}} \le e^{dG_{\beta, r}(z)}\,.
$$

We are interested in this inequality for just one particular $z=1$. 
Now we use \eqref{beta} to have $e^{dG_{\beta, r}(1)}\asymp 1$.


Hence we proved that  for $p\ge 2$
$$
\frac{\||\nabla F(1, \cdot)|\|_p}{d^{\beta/2}} \le C\,.
$$

Exactly the same reasoning shows that for
 $1<p<2$
$$
\frac{\||\nabla F(1, \cdot)|\|_p}{d^{\beta/p}} \le C\,.
$$
Theorem \ref{RXf} is completely proved just by plugging formula \eqref{beta} for $\beta$.
\end{proof}

\begin{proof}
The proof of theorem \ref{Xf} follows the same lines, but we need to use Theorem \ref{mp} for Banach spaces $X^*$. This is where  we use that
if $X^*$ is of type $2$ then is of finite co-type by K\"onig--Tzafriri theorem 7.1.14 in \cite{HVNVW2}.  Type $2$  is needed to conclude (using Khintchine--Kahane's inequality, see e.g.  \cite{HVNVW2}):
$$
\bE_\eps\Bigg\| \sum\eps_j \la_j\Bigg\|_p \le C_p\bE_\eps\Bigg\| \sum\eps_j \la_j\Bigg\|_2  \le C\big(\sum\|\la_j\|_{X^*}^2)^{1/2} \le C\,.
$$

\end{proof}

\subsection{Comparison of $\Delta^{1/2}\cdot$ and $|\nabla\cdot|_X$}
\label{co}

The reader can notice that in this paper we are mostly interested in bounding the expressions of the type $|\nabla f|_X$. It would be interesting to 
get from this the estimates of the  expressions of the type $\Delta^{1/2} f$.  But for Banach space valued functions it is mostly an open task.

For Banach space valued functions $f:\Om_n\to X$ it is not quite clear who majorized whom if we deal with $\|\Delta^{1/2} f\|_{L^p(X)}$ and
$\||\nabla f|_X\|_p$.

We would like to decide for exactly what class of Banach spaces $X$
$$
\|\Delta^{1/2} f\|_{L^p(X)} \le C_p \||\nabla f|_X\|_p\,.
$$
We think that this is the class of spaces of finite co-type. In fact, this is one way to express the boundedness {\it from below} of Riesz transform of Hamming cube 
in spaces $L^p(\Om_n, X)$, $1<p<\infty$. For $X=\bR$ this boundedness from below is always true, see \cite{BELP}.

On the other hand, the converse inequality, that is the boundedness of Riesz transform of Hamming cube {\it from above},
$$
  \||\nabla f|_X\|_p \le C_p \|\Delta^{1/2} f\|_{L^p(X)}
$$
does not have  a reasonably wide class of Banach spaces for which it holds. 
For $p \ge 2$ and $X=\bR$ this boundedness from above holds, but for $1<p<2$ it fails, see \cite{BELP}.

It can fail for very nice UMD space $X$ even for $p>2$.


\section{Type $2$ Banach spaces in Theorem \ref{Xf}}
\label{2sm}

By $L^p(X)$ we always mean $L^p(\Om_n; X)$, where $\Om_n$ is Hamming cube.
Let $1/q+1/p=1$.
Let $\cP(d, X)$ be the collection of polynomials with coefficients in Banach space $X$ and of degree at most $d$.
We prove that Theorem \ref{Xf} can  be somewhat strengthened in the sense of the  power of $d$.
Next theorem deals with $X$ such that $X^*$ is of type $2$. In particular, $X^*$ and $X$ are $K$-convex. Let $\pi\al$ denote the angle from Theorem \ref{pisier}.
\begin{thm}
\label{main}
If $f\in \cP(d, X)$ and $X^*$ is of type $2$.
Then $\||\nabla f|_X\|_{L^p(X)}\le C d^{\frac{2-\al}{p}} \| f\|_{L^p(X)}$ for $1<p<2$, and  $\||\nabla f|_X\|_{L^p(X)}\le C d^{1-\frac{\al}2} \| f\|_{L^p(X)}$ for $p\ge 2$.
\end{thm}
\begin{proof}
We recall the formula from \cite{IVHV}:
\begin{equation}
\label{formula}
D_j e^{-t\Delta} f (\eps)= \frac{e^{-t}}{(1-e^{-2t})^{1/2}} \bE_\xi\Big( \delta_j (t) f(\eps_1\cdot\xi_1(t), \dots, \eps_n\cdot\xi_n(t))\Big)\,.
\end{equation}
Here 
$$
\delta_j(t) := \frac{\xi_j(t)-e^{-t}}{(1-e^{-2t})^{1/2}},
$$
where $\xi_j(t)$ are independent random variables having values $\pm 1$ with probabilities $\frac{1\pm e^{-t}}{2}$.

The symmetric counterpart is
$$
\delta_j'(t) := \frac{\xi_j(t)-\xi_j'(t)}{(1-e^{-2t})^{1/2}},
$$
where vector $\{\xi_j'(t)\}$ is independent copy of $\{\xi_j(t)\}$.

We use the notation $\ell^2_n$ for $\ell^2_n(X^*)$ with the norm $(\sum_{j=1}^n \|\la_j||_{X^*}^2)^{1/2}$.
From \eqref{formula} for every $\eps \in \Om_n$ we can write
$$
|\nabla e^{-t\Delta} f(\eps)| =\frac{e^{-t}}{(1-e^{-2t})^{1/2}} \max_{\la: \|\la\|_{\ell^2_n} =1} \Big(\bE_\xi  \Big|\sum_{j=1}^n \delta_j(t) \langle \la_j, f(\eps\cdot\xi)\rangle\Big|\Big)\,.
$$
Hence,
$$
|\nabla e^{-t\Delta} f(\eps)|\le  \frac{e^{-t}}{(1-e^{-2t})^{1/2}}\max_{\la: \|\la\|_{\ell^2_n} =1}\bE_\xi  \Big|\sum_{j=1}^n   \delta_j (t) \langle \la_j, f(\eps\cdot\xi(t))\rangle\Big| =
$$
$$
 \frac{e^{-t}}{(1-e^{-2t})^{1/2}} \max_{\la: \|\la\|_{\ell^2_n} =1}\Big(\bE_\xi  \Big|\langle\sum_{j=1}^n \delta_j(t)  \la_j, f(\eps\cdot\xi)\rangle\Big|\Big)\le
 $$
 $$
 \frac{e^{-t}}{(1-e^{-2t})^{1/2}} \max_{\la: \|\la\|_{\ell^2_n} =1}\bE_\xi  \Big(\Big\|\sum_{j=1}^n \delta_j(t)  \la_j\Big\|_{X^*} \|f(\eps\cdot \xi)\|_X\Big)\le
 $$
 $$
 \frac{e^{-t}}{(1-e^{-2t})^{1/2}} \max_{\la: \|\la\|_{\ell^2_n} =1}\Big( \bE_\xi  \Big\|\sum_{j=1}^n \la_j  \delta_j (t) \Big\|^q \Big)^{1/q}\, \Big(\bE_\xi \|f(\eps\cdot\xi)\|^p\Big)^{1/p}\,.
 $$
\bigskip
We wish to prove that if $q>2$ then
\begin{equation}
\label{2q}
\||\nabla e^{-t\Delta} f(\eps)|\|_p \le  \frac{C(q)e^{-t}}{(1-e^{-2t})^{1/p}} \|f\|_{L^p(X)}\,,
\end{equation}
and if
$1\le q\le 2$ then
\begin{equation}
\label{q2}
\||\nabla e^{-t\Delta} f(\eps)|\|_p \le  \frac{C(q)e^{-t}}{(1-e^{-2t})^{1/2}} \|f\|_{L^p(X)}\,,
\end{equation}

For that let us work now with the term $\bE_\xi  \Big\|\sum_{j=1}^n \la_j  \delta_j (t) \Big\|_{X^*}^q$ for a fixed $\{\la_j\}\in \ell^2_n$ of norm $1$.



\bigskip

$$
B^q:=\bE_\xi  \Big\|\sum_{j=1}^n \la_j  \delta_j (t) \Big\|^q \le   \bE_{\xi,\xi'} \Big\|  \sum_{j=1}^n \la_j  \delta_j' (t) \Big\|^q =
$$
$$
 \bE_{\xi,\xi'}\bE_{r} \Big\|  \sum_{j=1}^n r_i\la_j  \delta_j' (t) \Big\|^q,
 $$
 where $r_i$ are independent Rademacher random variables.
\bigskip

The next lemma was provided by A. Borichev.
\begin{lem}
\label{aplusb}
Let $a, b\ge 0$ and $Q\ge 2$ be a large number. Then
$$
(a+b)^Q \le 6 a^Q + Q^Q b^Q\,.
$$
\end{lem} 
\begin{proof}
We need to show that  for all positive $t$, $(t+1)^Q \le 6 t^Q +Q^Q$. If $t \le Q-1$ this is immediate.
If $t\ge Q-1\ge 1$, we write 
$$
(t+1)^Q = t^Q \Big(\frac{t+1}{t}\Big)^Q  \le  t^Q \Big(\frac{t+1}{t}\Big)^{t+1} \le 2t^Q \Big(1+\frac1t\Big)^{t} \le 2e t^Q\,.
$$
\end{proof}

We continue to estimate $B^q$:
$$
B^q \le   \bE_{\xi,\xi'}\bE_{r} \Big\|  \sum_{j=1}^n r_i\la_j  \delta_j' (t) \Big\|_{X^*}^q \le C_q \bE_{\xi,\xi'}\Big(\bE_{r} \Big\|  \sum_{j=1}^n r_i\la_j  \delta_j' (t) \Big\|_{X^*}^2\Big)^{q/2} \le
$$
$$
C_q D^{q/2} \bE_{\xi,\xi'} \Big(\sum_{j=1}^n |\delta_j'(t)|^2 \|\la_j\|_{X^*}^2\Big)^{q/2}\,.
$$
In the last inequality we used that $X^*$ is of type $2$. The penultimate inequality is Kahane--Khintchine's inequality, see \cite{HVNVW2}.

Notice that if $1\le q\le 2$ then the above inequality
gives
$$
B^q \le C_q D^{q/2} \Big(\bE_{\xi,\xi'} \sum_{j=1}^n |\delta_j'(t)|^2 \|\la_j\|_{X^*}^2\Big)^{q/2}
$$
Hence,
\begin{equation}
\label{qle2}
1\le q\le 2\Rightarrow B \le C_q D^{q/2} \Big( \sum_{j=1}^n\|\la_j\|_{X^*}^2\Big)^{1/2} \le   C_q D^{q/2} \,.
\end{equation}

The estimate in case $2<q<\infty$ is much more interesting.

\bigskip

Now we will continue by thinking that $q$ is an  even integer, $q=2k$ (it is not important, just convenient). Let us now  estimate
\begin{equation}
\label{E}
E:=\bE_{\xi, \xi'}  \Big(\sum_{j=1}^n |\delta_j'(t)|^2 \|\la_j\|^2 \Big)^{k}
\end{equation}
We  denote
\begin{equation}
\label{fj}
f_j := |\delta_j'(t)|^2\|\la_j\|^2\,.
\end{equation}
Below we use Lemma \ref{aplusb} with $Q:=q/2-1=k-1$:
\begin{eqnarray*}
&\bE_{\xi, \xi'} (\sum_{j=1}^n f_j)^k =\bE_{\xi, \xi'} \sum_{i=1}^n(\sum_{j=1}^n f_j)^{k-1}  f_i=
\\
&\bE_{\xi, \xi'} \sum_{i=1}^n( (f_1+\dots f_{i-1}+f_{i+1}+\dots f_n) + f_i)^{k-1}  f_i \le ^{Lemma \,\ref{aplusb}}
\\
& (k-1)^{k-1} \bE_{\xi, \xi'}  \sum_{i=1}^n f_i^k +6  \bE_{\xi, \xi'}  \sum_{i=1}^n (f_1+\dots f_{i-1}+f_{i+1}+\dots f_n)^{k-1} \bE_{\xi, \xi'} f_i \le
\\
&(k-1)^{k-1}  \bE_{\xi, \xi'} \sum_{i=1}^n f_i^k +6 \bE_{\xi, \xi'}  (\sum_{j=1}^n   f_j)^{k-1}  \bE_{\xi, \xi'}  (\sum_{j=1}^n   f_j)\le
\\
&(k-1)^{k-1}  \bE_{\xi, \xi'} \sum_{i=1}^n f_i^k  +6(k-2)^{k-2}\bE_{\xi, \xi'} \sum_{i=1}^n f_i^{k-1} \bE_{\xi, \xi'}  (\sum_{j=1}^n   f_j) +
\\
& 6^2 \bE_{\xi, \xi'}  (\sum_{j=1}^n   f_j)^{k-2}  \big[\bE_{\xi, \xi'}  (\sum_{j=1}^n   f_j)\big]^2 \le
\\
&(k-1)^{k-1}  \bE_{\xi, \xi'} \sum_{i=1}^n f_i^k  +\dots + 6^\ell (k-\ell)^{k-\ell} \bE_{\xi, \xi'}  (\sum_{j=1}^n   f_j^{k-\ell}  )\big[\bE_{\xi, \xi'}  (\sum_{j=1}^n   f_j)\big]^{\ell}+\dots
\\
&+6^{k-1} \big[\bE_{\xi, \xi'}  (\sum_{j=1}^n   f_j)\big]^{k}\,.
\end{eqnarray*}
We used the fact that $\delta_j'(t)$, $j=1,\dots, n$, are independent exactly as this has  been done in Rosenthal's \cite{HR}.

\medskip

Now coming back to our notation \eqref{fj} we see that as 
\begin{equation}
\label{moments}
\bE|\delta_j'(t)|^2  =2,\,\quad \bE |\delta_j'(t)|^{m} \le \frac{2^{m-1} }{\sqrt{1-e^{-2t}}^{m-2}}\,.
\end{equation}
$$
\bE_{\xi, \xi'}  (\sum_{j=1}^n   f_j )\le 2 \|\{\la_j\}\|_{\ell^2_n}^2,
$$
$$
\bE_{\xi, \xi'}  (\sum_{j=1}^n   f_j^{k-\ell} )\le  \frac{ 2^{2k-2\ell}}{\sqrt{1-e^{-2t}}^{2k-2\ell-2}}\sum_{j=1}^n \|\la_j\|^{2k-2\ell}\,.
$$

\bigskip

Therefore, we can estimate $E$ from \eqref{E} as follows:
$$
E \le 24^k \sum_{\ell=0}^{k-2} \Big[  \frac{(k-\ell)^{k-\ell}}{\sqrt{1-e^{-2t}}^{2k-2\ell-2}}\sum_{j=1}^n \|\la_j\|^{2k-2\ell} \|\{\la_j\}\|^{2\ell}\Big] +24^k \|\{\la_j\}\|^{2k}\,.
$$

This obviously gives
$$
E \le  2(24)^k \sum_{\ell=0}^{k-2} (k-\ell)^{k-\ell}  \Big(\frac{1}{\sqrt{1-e^{-2t}}}\Big)^{2k-2\ell-2}\Big(\sum_{j=1}^n \|\la_j\|^{2}\Big)^{k-\ell} \|\{\la_j\}\|^{2k-2}\,+
$$
$$
24^k \|\{\la_j\}\|^{2k}\,.
$$
And so,
$$
E\le C'(q) \|\{\la_j\}\|^q  \sum_{\ell=0}^{k-2}  (k-\ell)^{k-\ell} \Big(\frac{1}{\sqrt{1-e^{-2t}}}\Big)^{2k-2\ell-2}
$$


Then
\begin{equation}
\label{qge2}
B \le \frac{C(q)}{(1-e^{-2t})^{\frac12-\frac1q}}  \|\{\la_j\}\|_{\ell^2_n(X^*)} =  \frac{C(q)}{(1-e^{-2t})^{\frac12-\frac1q}}  \,.
\end{equation}


 \bigskip

Now let us use  \eqref{2q} for $1<p<2$ and \eqref{q2} for $p\ge 2$ to finish the proof.
We can consider $x=e^{-t}$ and write those inequalities as the estimate of $p$-th norm of 
$$
F_f(x, \eps):= \sum_S x^{|S|} \hat  f(S)\eps_S, \quad\text{where}\,\, f= \sum_S  \hat  f(S)\eps_S, \,\,0\le x\le 1\,.
$$
We get from \eqref{qge2} and \eqref{qle2} correspondingly that
\begin{equation}
\label{Banachp2}
\||\nabla F(x, \cdot)|\|_{p} \le \frac{|x|}{(1-x^2)^{1/p}} \|f\|_{L^p(X)} , \quad \forall x\in [-1, 1], 1<p<2\,.
\end{equation}
\begin{equation}
\label{Banach2p}
\||\nabla F(x, \cdot)|\|_{p}  \le \frac{|x|}{(1-x^2)^{1/2}} \|f\|_{L^p(X)} , \quad \forall x\in [-1, 1], p\ge 2\,.
\end{equation}
We initially have this estimates only for $0\le x\le1$ but flipping $x\to -x$ is absorbed by flipping $\eps\to -\eps$.

Now consider a new function in the complex domain:
$$
H(z):= \log \, \||\nabla F(z, \cdot)|_X\|_p\,.
$$
We repeat verbatim the reasoning of Section \ref{complex} but instead of domain $\Om\setminus \Om(\beta, r)$ and its Green's function, we consider domain $\bC\setminus [-1+\frac1{d^2}, 1-\frac1{d^2}]$, whose Green's function $G_d$ satisfies
$$
G_d(1)\asymp \frac1d\,.
$$
This proves
$$
\frac{\||\nabla F(1, \cdot)|_X\|_p}{d^{\max (2/p, 1)}} \le C\,.
$$

\bigskip

But \eqref{qge2} and \eqref{qle2} can be used more efficiently if we use Pisier's Theorem \ref{pisier} again. In fact, it can be used. As $X^*$ has type $2$, it is $K$ convex.  Then $X$ is $K$-convex.
Let us fix $\beta$ to be chosen later and consider domain
$$
O_{\beta, \al}:=
\Big(1-\frac1{d^\beta}\Big)\, O_\al,
$$
where $O_\al$ was introduced in the previous Section.

\medskip

As $X$ is $K$-concave, so is $\ell^2(X)$.
Consequently \eqref{qge2} and \eqref{qle2} and Pisier's Theorem \ref{pisier} applied to $L^p(\Om_n, \ell^2(X))$ show that

$$
\||\nabla F(z, \cdot)|_X\| \le Cd^{\beta/p}, \,1<p<2, \quad \le C d^{\beta/2},\, p\ge 2, \quad z\in O_{\beta, \al}\,.
$$
Let $G_{\beta, \al}$ denote Green's function of $\bC\setminus \bar{O}_{\beta, \al}$.

We repeat verbatim the reasoning of Section \ref{complex} but instead of domain $\Om\setminus \Om(\beta, r)$ and its Green's function, we consider domain $\bC\setminus \bar{O}_{\beta, \al}$ whose Green's function $G_{\beta, \al}$ satisfies
$$
G_{\beta, \al}(1)\asymp  \Big(\frac1{d^\beta}\Big)^{\frac{\pi}{2\pi-\pi\al}}\asymp \frac1d\,,
$$
if 
$$
\beta= 2-\al\,.
$$
This proves
$$
\frac{\||\nabla F(1, \cdot)|_X\|_p}{d^{\max (\frac{2-\al}{p}, \frac{2-\al}{2})}} \le C\,.
$$

Theorem \ref{main} is proved.
\end{proof}

\begin{rem}
\label{DN}
We already mentioned in Section \ref{co} that
the boundedness of Riesz transform of Hamming cube {\it from above},
$$
  \||\nabla f|_X\|_p \le C_p \|\Delta^{1/2} f\|_{L^p(X)}
$$
does not have  a reasonably wide class of Banach spaces for which it holds. 
For $p \ge 2$ and $X=\bR$ this boundedness from above holds, but for $1<p<2$ it fails, see \cite{BELP}.
It can fail for very nice UMD space $X$ even for $p>2$.
Therefore, Theorem \ref{KX} or other Bernstein--Markov type
estimates of $\Delta^{1/2}f$ in \cite{EI1}   for $X$-valued polynomials $f$ with $X$ being $K$-convex
cannot help to prove the estimates of the type of Theorem \ref{main} or Theorem \ref{Xf}.
\end{rem}


\section{Non-commutative random variables  and Bernstein--Markov inequalities on Hamming cube}
\label{FLP}

We wish to demonstrate how the technique of non-commutative random variables can be used to
prove certain Bernstein--Markov inequalities on Hamming cube. The  estimates below are not as good as in the previous Section, and what follows
serves only illustrative purpose of showing a beautiful approach. 

To the best of our knowledge this approach was introduced by Francoise Lust-Piquard in \cite{FLP}, \cite{BELP}. All results of those papers are commutative, all methods are non-commutative. And even though many non-commutative proofs of those papers are by now made commutative (see, e.g. \cite{ILVHV}), still some non-commutative proofs did not get a commutative analogs up to now.

The non-commutative proof of a certain Bernstein--Markov inequality below  is given not because of its efficiency, but because of its beauty.

\bigskip

We will prove now that for $p\ge 2$
\begin{equation}
\label{ncBM}
f:\Om_n\to \bR, \deg f\le d\Rightarrow \||\nabla f|\|_p \le C_p\,d \|f\|_p,
\end{equation}
which is worse than Theorem \ref{RXf}.

Let
$$
Q=\begin{bmatrix}
0 & 1\\
1 & 0
\end{bmatrix},\,\, P = \begin{bmatrix}
0 & i\\
-i & 0
\end{bmatrix},\,\, U= i QP,
$$
They have anti-commutative relationship
\begin{equation}
\label{antiC}
QP=-PQ\,.
\end{equation}
Let $Q_j= I\otimes \dots Q\otimes I\dots\otimes I $, $P_j= I\otimes \dots P\otimes I\dots\otimes I $, on $j$-th place.
These are independent non-commutative random variables in the sense of $\text{tr}$ = sum of diagonal elements divided by $2^n$.

Put $Q_A = \Pi_{i\in A} Q_i$, $P_A = \Pi_{i\in A} P_i$

Now one considers algebra generated by $Q_j, P_j$ (this is algebra of all matrices $\mathcal{M}_{2^n}$). 
We have a projection $\mathcal{P}$ from multi-linear polynomials in $P_j, Q_j$ (notice $P^2=I, Q^2=I$) that kills everything except terms having only $Q's$.

Small (really easy) algebra shows (see \cite{BELP}) that $\mathcal{P}$ can be written as $\rho\, Diag \,\rho^*$, where $\rho$ is a  conjugation by a unitary operator, and $Diag$, is an operator on matrices that just kills all matrix elements except the diagonal. This $Diag$ is obviously the contraction on Schatten-von Neumann class $S_p$ for any $p\in [1, \infty]$ (obvious for Hilbert--Schmidt, $p=2$,  class and for bounded operators--so interpolation does that).

$$
\mathcal{R}(\theta)Q_A = \Pi_{j\in A} ( Q_j\cos\theta+ P_j\sin\theta)\,,\,\, \mathcal{R}(\theta)P_A = \Pi_{j\in A} ( P_j\cos\theta - Q_j\sin\theta)\,.
$$
One can easily check that the action of $\mathcal{R}(\theta)$ is $ R(\theta)^* T R(\theta)$ where $R(\theta)$ is a unitary matrix which is $n$-fold tensor product of
$$
\rho_\theta=\begin{bmatrix}
1 & 0\\
0 & e^{i\theta}
\end{bmatrix}
$$
Extend it by linearity onto the whole algebra $\mathcal{M}_{2^n}$. Then it is obvious that automorphism $\mathcal{R}(\theta)$ preserves all Schatten--von Neumann $S_p$ norms.

For any $f=\sum_{A\subset [n]} \hat f(A) \eps^A$, the reasoning of \cite{BELP} dictates to assign a non-commutative object, a matrix from $\mathcal{M}_{2^n}$ given by
 $$
 T_f =\sum_{A\subset [n]} \hat f(A) Q_A\,.
 $$
 Such matrices form commutative sub-algebra $M_{2^n}\subset \mathcal{M}_{2^n}$. Operators $\partial_j, D_j$ can be considered on $M_{2^n}$, acting in a canonical way. For example,
 $$
\partial_i Q_A =\begin{cases} Q_{A\setminus i},\,\,\text{if}\,\, i\in A;
\\
0, \, \,\text{if} \,\, i\notin A\,.\end{cases} 
 $$
 And $D_i:= \eps_i \pd_i$.

 \medskip


Consider now a matrix valued function 
$$
A_f(\theta)= \cR(\theta)T_f\,.
$$
It is a trigonometric polynomial of degree at most $d$ with matrix coefficients. Bernstein--Markov inequality (its proof) works for such matrix valued polynomials in exactly the same way as for scalar polynomials. The easiest way to see that is to prove Bernstein--Markov estimate by convolution with Fejer kernels. Then we get
$$
\Big\|\frac{d}{d \theta} A_f(\theta) \Big\|_{S_p} \le 2d \,\Big\|A_f(\theta) \Big\|_{S_p}, \quad 1\le p\le \infty.
$$

\bigskip

 On the other hand we can calculate easily $\frac{d}{d \theta} \cR(\theta) (Q_A)=$
 $$
 -\sum_{j\in A}\Pi_{i\in A, i<j} (\cos \theta Q_i+\sin\theta P_i) ((-\sin\theta Q_j+\cos\theta P_j)\Pi_{i\in A, i>j} (\cos \theta Q_i+\sin\theta P_i)\,.
 $$
 By commutativity relations between $P_j, Q_i$, we observe that this is nothing else but $\cR(\theta)\big(P_j \pd_j Q_A\big)$.
 Hence
 \begin{equation}
 \label{formulaNC}
 \frac{d}{d \theta} A_f(\theta) = \frac{d}{d \theta} \cR(\theta) T_f = -\cR(\theta)\Big( \sum_{j=1}^n P_j\pd_j T_f\Big)\,.
 \end{equation}
 Therefore,
 $$
 \|\cR(\theta)\Big( \sum_{j=1}^n P_j\pd_j T_f\Big)\|_{S_p} \le 2d\, \|\cR(\theta) T_f\|_{S_p}\,.
 $$
 Transformation $\cR_\theta$ preserves $S_p$ norms (see above), and so
 \begin{equation}
 \label{Tf}
\| \sum_{j=1}^n P_j\pd_j T_f\|_{S_p} \le 2d\, \|T_f\|_{S_p}\,.
\end{equation}
Let $\eps^{(k)}_j= -1$ if $j=k$ and $=1$ otherwise.
Following  \cite{BELP} we see  that $\| \sum_{j=1}^n P_j\pd_j T_f\|_{S_p} = \| \sum_{j=1}^n \eps^{(k)}_jP_j\pd_j T_f\|_{S_p} $. This is because
$$
Q_k \big( \sum_{j=1}^n P_j\pd_j \big) Q_k = \sum_{j=1}^n \eps^{(k)}_jP_j\pd_j 
$$
by anti-commutative relation $PQ=-QP$. Hence, for any sequence of signs 
$$
\| \sum_{j=1}^n P_j\pd_j T_f\|_{S_p} = \| \sum_{j=1}^n \eps_jP_j\pd_j T_f\|_{S_p}\,.
$$
Now one should use a non-commutative Khintchine inequality  of Lust-Piquard and Pisier \cite{FLPGP} and $2\le p\le \infty$:
$$
 \bE_\eps\| \sum_{j=1}^n \eps_jP_j\pd_j T_f\|_{S_p} \asymp_p \|\big( \sum_{j=1}^n (\pd_j T_f)^* P_j^*P_j\pd_j T_f\big)^{1/2}\|_{S_p}  + 
 $$
 $$
\big( \sum_{j=1}^n P_j\pd_j T_f(\pd_j T_f)^* P_j^*\big)^{1/2}\|_{S_p} \,.
 $$
 But $P_j^*P_j=P_j^2=I$, and in the second term $P_j$ and $\pd_j T_f$ commute (as there is identity matrix on the $j$-th place of $\pd_j T_f$).
 Therefore
 $$
\| \big(\sum_{j=1}^n (\pd_j T_f)^* P_j^*P_j\pd_j T_f\big)^{1/2}\|_{S_p}  + 
 \big(\sum_{j=1}^n P_j\pd_j T_f(\pd_j T_f)^* P_j^*)^{1/2}\|_{S_p}  = 
  $$
  $$
  2 \|\big(\sum_{j=1}^n (\pd_j T_f)^* \pd_j T_f\big)^{1/2}\|_{S_p}\,.
$$
Using \eqref{Tf} we conclude that for $p\in [2, \infty)$
$$
 \|\big(\sum_{j=1}^n (\pd_j T_f)^* \pd_j T_f\big)^{1/2}\|_{S_p}\le C_p d \|T_f\|_{S_p}\,.
 $$
 Both matrices in the left hand side and the right hand side are form commutative algebra $M_n$.  They are $T_f$ and $T_{|\nabla f|}$. It is left to notice that
 for any scalar function $f$ on $\Om_n$ we have $\|T_f\|_{S_p}= \|f\|_{L^p(\Om_n}$. This is just by using the basis of characteristic function of point sets $\{\eps\}$ on $\Om_n$ to compute the $S_p$ norm   of $T_f$. This basis consist of eigenfunctions of $T_f$ with eigenvalues $f(\eps)$. This is  easy, see in \cite{BELP}.
 
 We finally proved \eqref{ncBM} by non-commutative approach of Francoise Lust-Piquard.
 
 \section{Addendum 1: Fourier coefficients of conformal map $\vf$}
 \label{add}
 
 We consider the domain
 $$
 O_\al:=  -G_\al\cup G_\al,
 $$
 where 
 $G_\al=\{w: w=e^{-z},  |\arg \,  z|\le \frac{ \pi\al}{2}\}$. It is not very difficult to write down the boundary of this domain (see Section \ref{add2} below, where we partially do this).
 Then one can notice that it consists of two real analytic curve $\Gamma_+, \Gamma_-$,  symmetric with respect to $\bR$ and forming
 angle $\pi\al$ at $-1, 1$. 
 
 Hence, the conformal map $\vf^{-1}: O_\al \to \bD$ can be extended to a slightly wider domain bounded by  real analytic curves $\gamma_+, \gamma_-$, such that $\gamma_+$ lies a bit higher than $\Gamma_+$ and meets $\Gamma_+$ at $\pm 1$, and forms angle $\tau \pi$ with $\Gamma_+$ at points $\pm 1$, where $\tau$ is a small strictly positive number. Symmetrically for $\gamma_-, \Gamma_-$.

 Then conformal  map $\vf$ is extended to domain $ \cR$ bounded by two symmetric real analytic curves, intersecting $\bT$ only  at $\pm 1$ and making angle $(\al+\tau)\pi$ with $\bT$ at those points.
 
 Then 
 $$
 c_{m}=
 \int_\bT \frac1{z^{m+1}} \vf (z) dz= -\frac1{m}\int_{\pd\cR} \vf (z) d \frac1{z^m}=\frac1{m}\int_{\pd\cR} \frac1{z^m} \vf '(z) dz
 $$
 
 Now we us that on $\pd\cR$, $\pm 1$ we have $|\frac1z| \le \frac1{1+a(\tau) |y|}$ for $z= x+iy$. We get
 $$
 |c_m|\lesssim \int_0^2 \frac1{(1+a y)^m } \frac1{y^{1-\al}} dy\lesssim  \int_0^2 e^{-a_1 m \, y}\,dy^\al  = \int_0^{2^\al} e^{-b (m^\al t )^{1/\al}}\, dt\,.
 $$
 The last integral is $\le \frac1{m^\al}\int_0^\infty e^{-b s^{\frac1{\al}}} \, ds \lesssim \frac1{m^\al}$.
 
 \section{Addendum 2: boundary of $O_\al$ and getting rid of $\eps$ in the proof of Theorem 6 of \cite{EI1}}
 \label{add2}
 
 We consider two domains
 $$
 \Om(r):=\{z\in \bC: \max\{ |z-i\sqrt{r^2-1}|, |z+i\sqrt{r^2-1}|\}<r\},
 $$
 and
 $$
 O_\al:=  -G_\al\cup G_\al,
 $$
 where 
 $G_\al=\{w: w=e^{-z},  |\arg \,  z|\le \frac{\pi\al}{2}\}$.  
 We would like to compare those two domains for 
 \begin{equation}
 \label{choice}
 \pi\al=2\arcsin\frac1r\,.
 \end{equation}
  The choice of $r$ is dictated by the fact that for this choice the angle 
 that the boundaries have at point $1$ are the same (and symmetrically at $-1$).
 
 It is not very difficult to write down the boundary of $O_\al$, we will do this now for its parts near points $\pm 1$.

 \bigskip
 
Define $a$ as follows $\tan \frac{\pi\al}{2} = \frac{\pi}{a}$.
 Let us consider $G_\al\cap \{\Re z\in [0, \frac{a}{2}]\}$. Consider  $ G_\al(a/2):=e^{-G_\al\cap\{ \Re z\in [0, \frac{a}{2}]\}}=\{w=u+iv = e^{-z}, \,z\in G_\al\cap \Re z\in [0, \frac{a}{2}]\}$.  It consists of arcs $S_t$ of the circles centered at point $(0,0)$ of radii $e^{-t}, 0\le t \le \frac{a}2$, and  each  arc is symmetric  (w.r. to $\bR$), and has angle $2 \arctan\frac{\pi}{a}t$. In particular, $S_{a/2}$ is a half-circle that intercepts $v$-axis at points $\pm e^{-a/2}$. The boundary of the domain $G_\al(a/2)$ consists of $S_{a/2}$ and of two real analytic symmetric (w.r. to $\bR$) arcs,  one of them $ \Gamma(a/2)$ (the one in $\bC_+$)  being given by parametric equation:
 $$
 \Gamma(a/2):\quad u= e^{-t}\cos\frac{\pi}{a}t,\quad v= e^{-t}\sin\frac{\pi}{a}t, \quad 0\le t\le a/2\,.
 $$
 
 We also have an interesting circle of radius $r:=\sqrt{1+\frac{a^2}{\pi^2}}=\frac1{\sin(\frac{\pi\al}{2})}$.
 with center at $-i\sqrt{r^2-1}=-i\frac{a}{\pi}=-i\,\text{cotan} (\frac{\pi\al}{2})$ .
 
 \bigskip
 
 Let us check that $\Gamma(a/2)$ lies below the circle, in other words that
 $$
 \Big( e^{-t}\cos\frac{\pi}{a}t\Big)^2 + \Big(e^{-t}\sin\frac{\pi}{a}t+ \frac{a}{\pi}\Big)^2 < 1+\frac{a^2}{\pi^2}, \, \text{for small}\,\, t>0\,.
 $$
 This is the same as
 $$
 e^{-2t} + 2 e^{-t} \frac{a}{\pi} \sin \frac{\pi}{a} t <1\,.
 $$
 We write
 $$
 1-2t +\frac{4t^2}{2}-\frac{8t^3}{6}+\dots +2 (1-t +\frac{t^2}{2}-\frac{t^3}{6}+...) ( t- ct^3+\dots)=
 $$
 $$
 1 +t^3-\frac43 t^3 - 2ct^3 +\dots<1,
 $$
 if $t$ is small as $c$ is positive.  So the lens domain $\Om(r)$ of \cite{EI1} seems to  contain $O_\al$, at least it is not contained in it as $\Gamma(a/2)$ lies inside $\Om(r)$.
 
 That represents a small problem for \cite{EI1} because inclusion (103) there is not valid if one chooses $r$ according to our preferred choice \eqref{choice}. In its turn this is reflected
 in the formulas for conformal mapping one uses around (103). But the formula for conformal mapping of the unit disc onto $\Om(r)$ is straightforward.
 
 \bigskip
 
 But if one chooses $r$ not according to \eqref{choice} but smaller, than the angle of the lens domain at $\pm 1$ is smaller than $\pi\al$ and inclusion (103)  holds.
 Thus Theorem 6 of \cite{EI1} reproves the heat smoothing result of \cite{MN} with $A(p, X) > \frac1\al$, where $\al$ is the angle from Pisier's Theorem \ref{pisier}. But one can notice that just a small improvement in \cite{EI1} reasoning gives the heat smoothing result with $A(p, X) =\frac1\al$.
 
 \bigskip
 
 Let us indicate this small change that should be implemented to get $A(p, X) =\frac1\al$ in Theorem 6 of \cite{EI1}.
 
 As, in the contrast to (103) of \cite{EI1}, we have $\Om_\al\subset \Om(r)$ with $r$ as in \eqref{choice}, then one need the estimates of conformal mapping of the disk onto $O_\al$ (the smaller of two domains).
 Of course the angle that boundary of $O_\al$ form at point $1$ (and $-1$) is just $\pi\al$ (in notations of \cite{EI1} it is $\theta$). This angle is the same for $\Om(r)$. But this observation is not enough to conclude the same asymptotic for conformal maps on these two domains.
 
 However, this is a not a real problem. It is easy to see that asymptotic is in fact the same. To see that one transforms $\Om_\al$ and $\Om(r)$ to strips by logarithmic map and then one uses Warschawski's estimate from \cite{W}. It shows that asymptotic is the same because one can easily compute that $\int^\infty \Theta'(u)^2/ \Theta(u)\,du$ converges, see \cite{W} for the explanation what is $\Theta(u)$ for strips. 
 
 \bigskip
 
 The heat smoothing conjecture of \cite{MN} claims that $A(p, X)=1$ for $K$-convex $X$, but it is still a conjecture.
 The important time is $t_0=\frac1{d^\alpha}$. The estimate of Theorem \ref{KXal}, or slightly strengthened estimate of Theorem 6 of \cite{EI1} or Theorem 5.1 of \cite{MN}, all those estimates show that if $X$ is $K$ convex, then  for $X$-valued $f$ in the $d$-tail space 
 $$
 \|e^{-t_0 \Delta}f \|_{L^p(X)} \le C \|f\|_{L^p(X)}\,.
 $$
 This does not give us any interesting information. What the heat smoothing conjecture basically says is the following, let $\alpha$ be the angle from Pisier's Theorem \ref{pisier}, then
  $$
 t_0=\frac1{d^\alpha} \Rightarrow \|e^{-t_0 \Delta}f \|_{L^p(X)} \le \eps(d) \|f\|_{L^p(X)}, \quad \eps(d)\to 0, \, d\to \infty\,.
 $$
 This is still open.

\end{document}